\documentclass[a4paper,12pt]{article}

\usepackage[T2A]{fontenc}
\usepackage[utf8]{inputenc}
\usepackage[english]{babel}

\usepackage[dvipsnames]{xcolor}
\usepackage{float}

\usepackage[a4paper,top=1cm,bottom=2cm,left=1.5cm,right=1.5cm,marginparwidth=1.75cm]{geometry}

\usepackage{amsmath}
\usepackage{amssymb}
\usepackage{graphicx}
\usepackage[colorlinks=true, allcolors=blue]{hyperref}
\usepackage{array}
\usepackage[
style=nature,
maxbibnames=3,
autolang=other,
]{biblatex}
\usepackage{csquotes}

\usepackage{amsthm}
\usepackage{epsfig}

\newtheorem{theorem}{Theorem}[section]

\newtheorem{lemma}[theorem]{Lemma}

\theoremstyle{definition}
\newtheorem*{definition}{Definition}
\newtheorem{ex}[theorem]{Example}

\def\NN{{\mathbb N}}
\let\geq\geqslant
\let\leq\leqslant

\def\HG{{\mathcal G}}
\def\gplus#1{\mathop{+\!\!_{_#1}}}
\def\gtimes#1{\mathop{\times\!\!_{_#1}}}
\def\hgn{\text{\rm HG}}
\def\cnst#1{\star #1}

\def\hats{{\sc Hats\ }}

\usepackage{tikz}
\usetikzlibrary{
  arrows,
  mindmap,
  trees,
  shadows,
  graphs,
  graphs.standard,
  decorations.markings,
  intersections,
  calc
}

\usepackage[]{xcolor}%
\usepackage{xargs}
\usepackage[colorinlistoftodos,prependcaption,textsize=small]{todonotes}
\newcommandx{\unsure}[2][1=]{\todo[linecolor=red,backgroundcolor=red!25,bordercolor=red,#1]{#2}}
\newcommandx{\change}[2][1=]{\todo[linecolor=blue,backgroundcolor=blue!25,bordercolor=blue,#1]{#2}}
\newcommandx{\info}[2][1=]{\todo[linecolor=OliveGreen,backgroundcolor=OliveGreen!25,bordercolor=OliveGreen,#1]{#2}}
\newcommandx{\improvement}[2][1=]{\todo[linecolor=Plum,backgroundcolor=Plum!25,bordercolor=Plum,#1]{#2}}
\newcommandx{\thiswillnotshow}[2][1=]{\todo[disable,#1]{#2}}

\tikzstyle{win_nodes}=[circle, draw=black, inner sep=1.5, fill=black]
\tikzstyle{win_edges}=[draw=black, thick]
\tikzstyle{lose_nodes}=[circle, draw=black, inner sep=1.5, fill=black]
\tikzstyle{lose_edges}=[draw=black, thick]

\definecolor{lgreen}{RGB}{180, 180, 180}
\tikzstyle{vecArrow} = [thick, decoration={markings,mark=at position
   1 with {\arrow[semithick]{open triangle 60}}},
   double distance=1.4pt, shorten >= 5.5pt,
   preaction = {decorate},
   postaction = {draw,line width=1.4pt, white,shorten >= 4.5pt}]
\tikzstyle{innerWhite} = [semithick, white,line width=1.4pt, shorten >= 4.5pt]

\title{The Hats game. The power of constructors}
\author{Aleksei Latyshev\thanks{ITMO University, St.Petersburg, Russia. Email:
    aleksei.s.latyshev@gmail.com} \and Konstantin Kokhas\thanks{St.Petersburg
    State University, St.Petersburg, Russia. Email: kpk@arbital.ru}}

\begin{document}
\maketitle

\begin{abstract}
  We analyze the following general version of the deterministic Hats game.
  Several sages wearing colored hats occupy the vertices of a graph. Each sage
  can have a hat of one of $k$ colors. Each sage tries to guess the color of his
  own hat merely on the basis of observing the hats of his neighbors without
  exchanging any information. A predetermined guessing strategy is winning if it
  guarantees at least one correct individual guess for every assignment of
  colors. We present an example of a planar graph for which the sages win for $k
  = 14$. We also give an easy proof of the theorem about the Hats game on
  “windmill” graphs.
\end{abstract}

\section{Introduction}

The hat guessing game goes back to an old popular Olympiad problem. Its
generalization to arbitrary graphs has recently attracted the interests of mathematicians (see, e.g. \cite{bosek19_hat_chrom_number_graph,
  he20_hat_guess_books_windm, alon2020hat}). The theory of the game is based on
methods of combinatorial graph theory.

In this paper, we consider the version of the hat guessing game where sages
located at graph vertices, try to guess the colors of their own hats. They
can see the colors of hats on the sages at the adjacent vertices only. The
sages act as a team, using a deterministic strategy fixed at the beginning.
If at least one of them guesses a color of his own hat correctly, we say
that the sages win.

Most studies of this game consider the version where each sage
gets a hat of one of $k$ colors. The maximum number $k$ for which the sages can
guarantee the win is called the \emph{hat guessing number} of graph $G$ and denoted by
$\hgn(G)$. The computation of the hat guessing number for an arbitrary graph is a hard
problem. Currently, it is solved only for few classes of graphs: for complete graphs,
trees (folklore), cycles~\cite{cycle_hats}, and pseudotrees~\cite{Kokhas2018}. Also, there are some results for ``books'' and ``windmills'' graphs, see ~\cite{he20_hat_guess_books_windm}.

N.~Alon et al.~\cite{alon2020hat} studied the relation between the
hat guessing number and other graph parameters.
The problem of bounds on the hat guessing numbers of planar graphs
was mentioned in the papers~\cite[Conjecture 4]{bosek19_hat_chrom_number_graph} and \cite[Question
5.2]{he20_hat_guess_books_windm}. %
At the moment, the maximum known hat guessing number for planar graphs is 12 (see \cite{he20_hat_guess_books_windm}).

In their previous papers~\cite{Kokhas2019cliques1rus_en}
and~\cite{Kokhas2019cliques2rus_en} (joint with V.\,Retinsky), the authors considered the
version of the game with a variable number of hats (i.e., when the number of possible colors can
differ from sage to sage). This version is not only of independent
interest, but opens a more flexible approach to the analysis of the classical hat guessing game,
because it has nontrivial techniques for building strategies.

In this paper, we continue to study the hat guessing game with a variable number of colors.
We show, using the constructors machinery, how to build a planar graph with hat
guessing number at least 14. Also, we give a quite simple proof of the windmills theorem from ~\cite{he20_hat_guess_books_windm}.

This paper is structured as follows.

In the second section, we give necessary definitions and notations, recall several
constructor theorems from~\cite{Kokhas2019cliques1rus_en}, and give a
simple proof of the windmills theorem from~\cite{he20_hat_guess_books_windm} as an example of using these constructors.

In the third section, we build an outerplanar graph with hat guessing number at
least 6 (example~\ref{ex:planar6}) and a planar graph with hat guessing number at least 14 (theorem~\ref{thm:planar6}).

\section{Constructors}

\subsection{Definitions and notations}

We use the following notations.

$G = \langle {V, E} \rangle$ is a visibility graph, i.\,e.\ graph with sages
in its vertices; we identify a sage with the a corresponding vertex.

$h\colon V\to \NN$ is a \emph{``hatness'' function}, which indicates the number of
different hat colors that a sage can get. For a sage $A\in V$, we call the
number $h(A)$ the \emph{hatness} of sage $A$. We may assume
that the hat color of the sage $A$ is a number from 0 to $h(A)-1$, or a residue modulo $h(A)$.

\begin{definition}
  The hat guessing game or \hats for short is a pair $\HG=\langle {G, h} \rangle$, where $G$
  is a visibility graph, and $h$ is a hatness function. So, sages are
  located at the vertices of the visibility graph $G$ and participate in a \emph{test}.
  During the test, each sage $v$ gets a hat of one of $h(v)$ colors. The sages
  do not communicate with each other and try to guess the colors of their own hats. If at least
  one of %
  their guesses is correct, the sages \emph{win}, or the game is
  \emph{winning}. In this case, we say that the graph is winning too, keeping in mind
  that this property depends also on the hatness function. Games in which the sages have no winning strategy are said to be \emph{losing}.

  Denote by $\langle {G, \cnst{m}} \rangle$ the classical hat guessing game where the  hatness function has constant value $m$. We say that a game $\HG_1 =
  \langle {G_1, h_1} \rangle$ \emph{majorizes} a game $\HG_2 = \langle {G_2,
    h_2} \rangle$, if $G_1=G_2$ and $h_1(v) \geq h_2(v)$ for every $v\in V$.
  Obviously, if a winning game $\HG_1$ majorizes a game $\HG_2$, then $\HG_2$ is
  a winning too. And vice versa: if a losing game $\HG_2$ is majorized by a
  game $\HG_1$, then $\HG_2$ is losing too.

\end{definition}

Let $m=\min\limits_{A\in V(G)} h(A)$. Then the game $\langle {G, h} \rangle$
majorizes game $\langle {G, \cnst{m}} \rangle$. In this case, if the game
$\HG = \langle {G, h} \rangle$ is winning, then $\hgn(G)\geq m$.

\subsection{Constructors}

By constructors we mean theorems that allow us to build new winning graphs by
combining graphs for which their winning property is already proved. Here are
several constructors from the papers~\cite{Kokhas2019cliques1rus_en,
  Kokhas2019cliques2rus_en}.

\begin{definition}
  Let $G_1=\langle {V_1, E_1} \rangle$, $G_2=\langle {V_2, E_2} \rangle$ be two
  graphs sharing a common vertex $A$. \emph{The sum of graphs $G_1$, $G_2$ with
  respect to vertex $A$} is the graph $\langle {V_1\cup V_2, E_1\cup E_2} \rangle$. We
  denote this sum by  $G_1\gplus{A}G_2$.

  Let $\HG_1 =\left\langle G_1, h_1 \right\rangle$, $\HG_2 =\left\langle G_2, h_2
  \right\rangle$ be two games such that $V_1\cap V_2 = \{A\}$. The game  $\HG =
  \langle {G_1 \gplus{A} G_2, h} \rangle$, where $h(v)$ equals $h_i(v)$ for $v\in
  V(G_i)\setminus\{A\}$ and $h(A)=h_1(A) \cdot h_2(A)$
  (fig.~\ref{fig:multiplication}), is called the \emph{product} of games  $\HG_1$,
  $\HG_2$ with respect to the vertex $A$. We denote the product by $\HG_1\gtimes{A}
  \HG_2$.
\end{definition}

\begin{figure}[H]
\kern-.5cm
  \centering
    \begin{tikzpicture}[
      scale=.7,
      every label/.append style = {font=\footnotesize},
      every node/.style = win_nodes,
      every edge/.append style = win_edges
      ]
      \foreach [count=\i] \x in {0, 6, 10, 14}
      \node[circle,draw,lightgray] (c\i) at (\x, 0) [minimum size=60pt] {};

      \foreach [count=\i] \x in {2, 4, 12, 12}
      {
        \coordinate (a\i) at (\x, 0);

        \draw[lightgray, fill=lightgray] (a\i) --
        (tangent cs:node=c\i,point={(a\i)},solution=1) --
        (tangent cs:node=c\i,point={(a\i)},solution=2) -- cycle;
      }

      \begin{scope}[every node/.style = {}]
        \node at (3, 0) {{$\times$}};
        \node at (8, 0) {{$=$}};
      \end{scope}

      \node[black, label=above:$h_1(A)$, label=below:$A$] (v) at (2, 0) {};
      \node[black, label=above:$h_2(A)$, label=below:$A$] (u) at (4, 0) {};
      \node[black, label=above:$h_1(A)\cdot h_2(A)$, label=below:$A$] (vu) at (12, 0) {};
    \end{tikzpicture}
  \caption{The product of games}
  \label{fig:multiplication}
\end{figure}
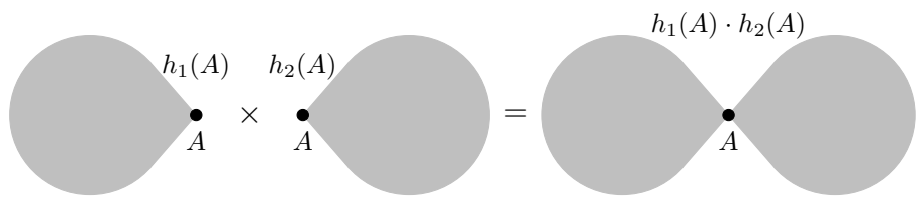

\begin{theorem}[on the product of games]
  \label{thm:multiplication}
  \cite[Theorem 3.1]{Kokhas2019cliques1rus_en}
  Let $\HG_1 = \langle {G_1^A, h_1} \rangle$ and $\HG_2 = \langle {G_2^A, h_2}
  \rangle$ be two games such that $V(G_1)\cap V(G_2)=\{A\}$. If sages win in
  the games $\HG_1$ and $\HG_2$, then they win also in game $\HG = \HG_1\gtimes{A}
  \HG_2$.
\end{theorem}

\begin{theorem}\label{thm:sum-lose}
  \cite[Theorem 4.1]{Kokhas2019cliques2rus_en}
  Let $G_1$ and $G_2$ be graphs such that $V(G_1)\cap  V(G_2)=\{A\}$,  $G=  G_1+_AG_2$.
  Let $\HG_1=\langle G_1, h_1\rangle$ and $\HG_2=\langle
  G_2, h_2\rangle $ be losing, $h_1(A)\geq h_2(A)=2$. Then
  the game $\HG=\langle G_1+_AG_2, h\rangle$ is losing, where
  $$
  h(x) =
  \begin{cases}
    h_1(x),&x\in V(G_1)\\
    h_2(x),&x\in V(G_2)\setminus A.
  \end{cases}
  $$
\end{theorem}

\begin{theorem}[On ``cone'' with vertex $O$ over graph $G$]\label{thm:konus}
  \cite[Theorem 4.5]{Kokhas2019cliques2rus_en}
  Assume the a game $\HG=\langle G, h \rangle$, where $V(G)=\{A_1, A_2, \ldots, A_k\}$,
  and $k$ games $\HG_i=\langle G_i, h_i\rangle$, $1\leq i\leq k$, are winning and the sets $V(G_i)$ are disjoint. In each graph $G_i$ one vertex is labeled $O$, and
  one of its neighbors is labeled $A_i$ such that the equality
  $h_1(O)=h_2(O)=\ldots=h_k(O)$ holds. Consider a new graph $G'=\langle{V(G'),
    E(G')}\rangle$, where
  $$
    V(G')=V(G_1)\cup\ldots \cup V(G_k),
    \qquad
    E(G')=E(G_1)\cup\ldots\cup E(G_k)\cup  E(G).
  $$
  Then game  $\langle G', h'\rangle$ is winning, where
  $$
  h'(x)=\begin{cases}
    h_i(x),         & \text{if $x$ belongs to one of the sets  \ } V(G_i)\setminus\{A_i\}, \\
    h_i(A_i)h(A_i), & \text{if $x$ coincides with $A_i$}.
  \end{cases}
  $$
\end{theorem}

To apply these constructor theorems we need ``bricks'', i.\,e.\ examples of
winning (or losing) graphs. The following theorem gives us a whole class of
such examples.

\begin{theorem}\label{thm:clique-win}
  \cite[Theorem 2.1]{Kokhas2019cliques1rus_en}
  Let $a_1$, $a_2$, \dots, $a_n$ be the hatnesses of $n$ sages located at the vertices of the complete graph. Then the sages win if and only if
  \begin{equation}
    \frac1{a_1}+\frac1{a_2}+\ldots+\frac1{a_n}\geq 1.
    \label{eq:clique-win}
  \end{equation}
\end{theorem}

\subsection{Application of constructors: windmills}

Let us show how constructors can be applied to study games with
constant hatness function. The following theorem was proved
in~\cite{he20_hat_guess_books_windm}. Using constructors makes it almost
obvious.

Let $k$ and $n$ be arbitrary positive integers.
Let $G_1$, $G_2$, \dots, $G_n$ be $n$ copies of complete graph $K_k$ in
each of which vertex $A$ is marked. A \emph{windmill} is a graph
$W_{k,n}=G_1\gplus{A} G_2\gplus{A}\ldots\gplus{A} G_n$. The vertex $A$ in this graph
is called the \emph{axis} of the windmill.

\begin{theorem}\label{thm:windmills}
  \cite[Theorem 1.4]{he20_hat_guess_books_windm}
  Let $k \geq 2$ and $n \geq \log_2(2k-2)$, then $\hgn(W_{k, n})=2k-2$.
\end{theorem}
\begin{proof}
  First, we prove that the game $\HG_1 = \langle {W_{k, n}, \cnst{2k-2}}
  \rangle$ is winning. Consider the complete graph $K_k$ in which one of vertices is  labeled $A$. Let $h$ be the hatness function such that $h(A)=2$, and the value of $h$ for the other vertices is equal to $2k-2$. The game $\HG=\langle {K_k, h} \rangle$ is winning by theorem~\ref{thm:clique-win}, because $\frac{1}{2k-2}\cdot (k -
  1) + \frac{1}{2} = 1$. Using theorem~\ref{thm:multiplication}, multiply $n$
  copies of $\HG$ with respect to the vertex $A$. We obtain a winning game $\HG_2=\langle {W_{k, n},
    h_2} \rangle$ on the windmill $W_{k,n}$ with axis $A$, where $h_2(A)=2^n\geq
  2k-2$ and $h_2(v)$ is equal to $2k-2$ for vertices $v\ne A$. The game $\HG_2$
  majorizes the game $\HG_1$, so the game $\HG_1$ is winning.

  Now, we prove that game $\HG_1' = \langle {W_{k, n}, \cnst{2k-1}} \rangle$ is losing.
  Label two vertices in the complete graph $K_k$ by $A$ and $B$ and consider the
  game  $\HG' = \langle {K_k, h'} \rangle$, where $h'(A)=2$, $h'(B)=2k-1$, and
  the values of $h'$ at the other vertices equal to $2k-2$. By
  theorem~\ref{thm:clique-win} the game $\HG'$ is winning. Using
  theorem~\ref{thm:sum-lose} multiply $n$ copies of game $\HG'$ with respect to the vertex $A$.
 We obtain a losing game $\HG_2'=\langle {W_{k,n},h_2'} \rangle$, where $h_2'(A)=2$
  and the values of $h_2'$ are equal to $2k-2$ or $2k-1$ at the other vertices. The game $\HG_2'$ is majorized by $\HG_1'$, so the game $\HG_1'$ is losing.

  The winning property of the game $\HG_1$ and the losing property of the game $\HG_1'$ together mean exactly that $\hgn(W_{k,  n}) = 2k-2$.

\end{proof}

The proof of the lower bound for $\hgn$ for theorem 1.5
from~\cite{he20_hat_guess_books_windm} is similar. Unfortunately, we do not have suitable constructors for the upper bound. This demonstrates once again that to prove that a game is losing is much harder than to prove that a game is winning.

\section{Planarity}

Recall that a graph is said to be \emph{planar}, if it can be drawn on the plane in such
a way that its edges intersect only at their endpoints. If there is an embedding in
the plane such that all vertices belong to the unbounded face of the embedding,
then the graph is is said to be \emph{outerplanar}.

The relation between the number $\hgn(G)$ and the planarity of graph $G$ is one
of the open problems. Namely, does there exist a planar graph with an arbitrary large
hat guessing number? In this section, we build an outerplanar graph $G$ with
$\hgn(G)\geq 6$ and planar graph $G$ with $\hgn(G)\geq 14$. Currently, we see no
approaches that could allow one to increase these numbers.

\subsection{Planarity and the product and cone constructors}

The product constructor (theorem~\ref{thm:multiplication},
fig.~\ref{fig:multiplication}), obviously, preserves planarity: we get a winning
planar graph by ``multiplying'' winning planar graphs.

\begin{definition}

The \emph{list of values} of a hatness function $h$ is the list $L(h)$ of all
values of $h$ arranged  in non-decreasing order. For example, the graph $G$ on
fig.~\ref{fig:planar6} has the list of values $L(G) = (6, 6, \dots, 6, 8)$.
\end{definition}

Hereafter we do not consider trivial games for which $\min h = 1$.

\begin{lemma}[on the second minimum]\label{thm:second-min}
  Let $\HG = \langle {G, h} \rangle$ be a winning game on planar graph $G$,
  $L(h) = (a_1, a_2, \dots)$. Then there exists a planar graph $G'$ with $\hgn(G')
  \geq a_2$.
\end{lemma}
\begin{proof}
  Let $A$ be the sage with hatness $a_1$. Pick $k$ such that $a_1^k \geq a_2$
  (this can be done, because $a_1 > 1$). Multiply $k$ copies of the game $\HG$ with respect ot the vertex $A$ using theorem~\ref{thm:multiplication}. We obtain a winning game $\HG'=\langle {G', h'} \rangle$, where $h'(A)=a_1^k\geq a_2$. Game $\HG '$
  majorizes the game $\langle {G', \cnst{a_2}} \rangle$ . So, the latter is winning
  too, and $\hgn(G')\geq \min h' = a_2$. Graph $G'$ is planar as a sum
  of planar graphs.
\end{proof}

Informally, this lemma means that searching for planar graphs with large hat
guessing number, we can ``allow'' one vertex to have small hatness (it is
reasonable to assign this vertex the smallest possible hatness, i.\,e.~2).
Similar statements are true for every graph property preserved under the
multiplication of graphs by a vertex (e.\,g. for outerplanarity).

The cone theorem, like the product theorem, can be used to build planar
graphs with large hat guessing number due to the following observation.

\begin{lemma}\label{thm:wheel}
  Let $\HG = \langle {G, \cnst{x}} \rangle$ be a winning game on outerplanar
  graph $G$, and let $\HG_1 = \langle {G_1, h_1} \rangle$ be a winning game on
a planar graph $G_1$, $L(h_1) = (a_1, a_2, a_3, \dots)$. Then there exists a graph
  $G'$ with $\hgn(G') \geq \min(a_2\cdot x, a_3)$.
\end{lemma}
\begin{proof}
  Apply theorem~\ref{thm:konus} for the game $\HG$ and the collection of games $\HG_i$, $1\leq
  i \leq |V(G)|$, where for every $i$ the game $\HG_i$ is a copy of game $\HG_1$.
  In each graph $G_i$ denote by $O$ and $A_i$ the vertices where the sages with hatness $a_1$ and $a_2$, respectively,c are located (see below fig.~\ref{fig:26666},
  example~\ref{ex:26666} and Subsection~\ref{subsec:planar14}). As a result, we
  obtain a winning game $\HG ' = \langle {G', h'} \rangle$. Graph $G'$ is planar,
  because  in this construction we can draw all graphs $G_i$ in the
  outer face of the  graph $G$.

The list $L(h')$ contains the values $a_1, a_3, a_4, \dots$ and $a_2\cdot x$. The
  second minimum in this list is $a_3$ or $a_2\cdot x$. So,  by
  lemma~\ref{thm:second-min}, there exists a planar graph
  with hat guessing number at least $\min(a_2\cdot x, a_3)$.
\end{proof}

This lemma shows that when building planar graphs with large hat guessing number,
the hatness function can have a relatively small second minimum. However, when applying the cone theorem, we have to compensate this deficiency using outerplanar graphs
with relatively large hat guessing number.

\begin{ex}\label{ex:26666}
  The game <<26666>> in fig.~\ref{fig:26666} is winning (the values of the
  hatness function are indicated near each vertex). It is obtained in the spirit of
  lemma~\ref{thm:wheel}. For this, we apply the cone theorem for the outerplanar graph $G$:
  we clue graph $G_1$ and its copy $G_2$ at the vertex $O$ and construct a copy of graph $G$ on the vertices $A_1$ and $A_2$ .
\end{ex}

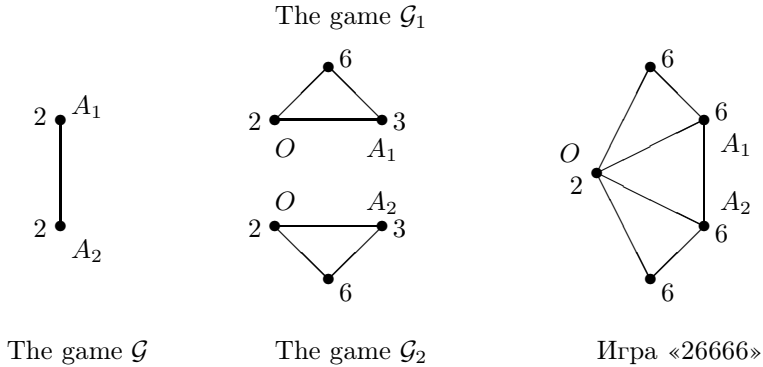
\begin{figure}[h]%
\setlength{\unitlength}{20bp}\footnotesize
\begin{center}
\begin{picture}(13,7)(0,-.5)%
    \put(1,2){\line(0,1){2}}%
    \put(1,2){\circle*{.2}}\put(1,4){\circle*{.2}}
    \put(1.2,1.4){$A_2$}   \put(1.2,4.1){$A_1$}
    \put(.5,1.8){2}        \put(.5,3.9){2}
    \put(0,-.5){The game $\HG$}
    \put(5,2){\line(1,0){2}}\put(5,2){\line(1,-1){1}}\put(7,2){\line(-1,-1){1}}
    \put(5,2){\circle*{.2}} \put(7,2){\circle*{.2}}  \put(6,1){\circle*{.2}}
    \put(5.,2.3){$O$}  \put(6.7,2.3){$A_2$}
    \put(4.5,1.8){2}    \put(7.2,1.8){3} \put(6.2,.6){6}
    \put(5,-.5){The game $\HG_2$}
    \put(5,4){\line(1,0){2}}\put(5,4){\line(1,1){1}}\put(7,4){\line(-1,1){1}}
    \put(5,4){\circle*{.2}} \put(7,4){\circle*{.2}}  \put(6,5){\circle*{.2}}
    \put(5.,3.3){$O$}  \put(6.7,3.3){$A_1$}
    \put(4.5,3.8){2}    \put(7.2,3.8){3} \put(6.2,5){6}
    \put(5,5.8){The game $\HG_1$}
    \put(11,3){\line(2,1){2}}\put(11,3){\line(1,2){1}}\put(13,4){\line(-1,1){1}}
    \put(11,3){\line(2,-1){2}}\put(11,3){\line(1,-2){1}}
    \put(13,2){\line(-1,-1){1}}\put(13,2){\line(0,1){2}}
    \put(11,3){\circle*{.2}} \put(13,4){\circle*{.2}}  \put(12,5){\circle*{.2}}
    \put(13,2){\circle*{.2}}  \put(12,1){\circle*{.2}}
    \put(10.3,3.2){$O$}  \put(13.3,3.4){$A_1$}  \put(13.3,2.3){$A_2$}
    \put(10.5,2.6){2}    \put(13.2,4){6}        \put(13.2,1.7){6}
    \put(12.2,5){6}      \put(12.2,.6){6}
    \put(11,-.5){Игра <<26666>>}
  \end{picture}%
\caption{An application of the cone theorem}\label{fig:26666}
\end{center}
\end{figure}

\begin{ex}\label{ex:planar6}
  Multiplying  three copies of the game <<26666>> (fig.~\ref{fig:26666}) by theorem~\ref{thm:multiplication}, we obtain a game ``Trefoil'', shown in
  fig.~\ref{fig:planar6}. This is an example of a winning outerplanar graph with hat
  guessing number at least 6. Due to computational experiments, we are sure that
  this hat guessing number is exactly 6, but we will not prove this here.
\end{ex}

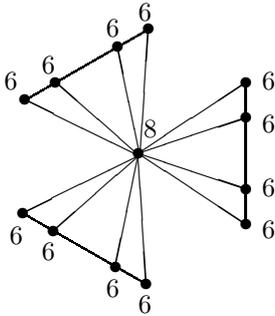
\begin{figure}[h]%
\setlength{\unitlength}{20bp}\footnotesize
\begin{center}
\begin{picture}(5,6)(-1.5,0)%
    \put(1,3){\line(3,2){2}}\put(1,3){\line(3,1){2}}
    \put(1,3){\line(3,-2){2}}\put(1,3){\line(3,-1){2}}
    \put(3,1.67){\line(0,1){2.66}}
    \put(3,4.33){\circle*{.2}} \put(3,3.67){\circle*{.2}}  \put(1,3){\circle*{.2}}
    \put(3,1.67){\circle*{.2}}  \put(3,2.33){\circle*{.2}}
\put(0.38,3.58){\rotatebox{120}{\put(1,3){\line(3,2){2}}\put(1,3){\line(3,1){2}}
    \put(1,3){\line(3,-2){2}}\put(1,3){\line(3,-1){2}}
    \put(3,1.67){\line(0,1){2.66}}
    \put(3,4.33){\circle*{.2}} \put(3,3.67){\circle*{.2}}
    \put(3,1.67){\circle*{.2}}  \put(3,2.33){\circle*{.2}}
    }}
\put(-1.11,5.31){\rotatebox{240}{\put(1,3){\line(3,2){2}}\put(1,3){\line(3,1){2}}
    \put(1,3){\line(3,-2){2}}\put(1,3){\line(3,-1){2}}
    \put(3,1.67){\line(0,1){2.66}}
    \put(3,4.33){\circle*{.2}} \put(3,3.67){\circle*{.2}}
    \put(3,1.67){\circle*{.2}}  \put(3,2.33){\circle*{.2}}
    }}
    \put(3.3,3.4){6} \put(3.3,4.2){6} \put(3.3,2.2){6} \put(3.3,1.4){6}
    \put(1,5.5){6}   \put(.4,5.2){6}  \put(-1.5,4.2){6}\put(-.8,4.5){6}
    \put(1,0){6}   \put(.4,.3){6}  \put(-1.4,1.3){6}\put(-.8,1.){6}
    \put(1.1,3.3){8}
  \end{picture}%
\caption{The game <<Trefoil>>}\label{fig:planar6}
\end{center}
\end{figure}

\subsection{An example of an arithmetic strategies on an almost complete graph}

 \begin{definition}
   An \emph{almost complete graph} is a complete graph without one edge,
   we denote it by $K_n^-$. If the vertices are numbered, we assume that the edge between the two last vertices is removed.
 \end{definition}

In~\cite{Kokhas2019cliques2rus_en} the authors proved theorem~\ref{thm:clique-win}
by demonstrating strategies based on arithmetic considerations. We use this approach to
prove that the following game on an almost complete graph is winning.

\begin{lemma}
The game $\HG = \langle {K_5^-, [2, 3, 14, 14, 14]} \rangle$ is winning.
\end{lemma}
\begin{proof}
Denote the sages by $A_2$, $A_3$, $A_{14}$, $B_{14}$, $C_{14}$, the subscripts show the
hatness, the edge $B_{14}C_{14}$ is absent. Denote the colors of sages' hats (given or
assumed) by $a_2$, $a_5$, $a_{14}$, $b_{14}$, $c_{14}$. Below, the calculations are modulo 42. For every hat arrangement consider, the sum
$$
S=21a_2 + 14a_3+3a_{14} \bmod 42.
$$
Let $M$ be the set of residues modulo 42. For every $x\in M$, the \emph{2-orbit} of $x$ is the set $\{x,x+21\}\subset M$, the \emph{3-orbit} is the
set $\{x,x+14,x+28\}\subset M$, and the \emph{14-orbit} is the set $\{x+3i:
i=0,1,2,\dots,14\}\subset M$. The sets $\mathcal B=\{0,1,2\}\subset M$ and $\mathcal
C=\{0,4,8\}\subset M$ are called \emph{traps}. It is obvious that these sets
meet every 14-orbit exactly at one point.

Now we describe the winning strategy for the sages. Sages $B_{14}$ and $C_{14}$ see hats
of sages $A_2$, $A_3$, and $A_{14}$ and can compute $S$. Let sage $B$ check the
hypothesis $S+3b_{14}\in\mathcal B$, and sage $C_{14}$ check the hypothesis
$S+3c_{14}\in\mathcal C$.

Sages $A_2$, $A_3$, and $A_{14}$ see the hat colors of $B_{14}$ and $C_{14}$,
and draw the conclusion, that if
$$
S\in(\mathcal B-3b_{14}), \quad\text{or }\quad  S\in(\mathcal C-3c_{14})
$$
 (the subtractions modulo 42), then $B_{14}$ or $C_{14}$ guess their own color
 correctly. Thus, if $S\notin (\mathcal B-3b_{14})\cup (\mathcal
 C-3c_{14})$, then sages $A_2$, $A_3$, and $A_{14}$ must guess correctly. For this,
 construct disjoint sets $\mathcal A_2$, $\mathcal A_3$, $\mathcal A_{14}$  such that
$$
\mathcal A_2 \cup \mathcal A_3 \cup \mathcal A_{14} \cup (\mathcal B-3b_{14})\cup (\mathcal C-3c_{14}) =M.
$$
Let every sage $A_i$, where $i=2$, 3, 14, choose an assumed color of his hat
so that $S\in\mathcal A_i$ under this assumption.
This can certainly be done if set $\mathcal A_i$ meets each $i$-orbit at exactly
one point.
For this, let $\mathcal{A}_2$ be an interval of the form $[x,x+20]$ consisting of 21 consecutive residues and $\mathcal{A}_3$ be an interval $[x,x+13]$.
Set $\mathcal A_{14}$ must consist of three numbers with
different residues modulo 3. In this way, if for the given hat arrangement the sum $S$ belongs to the set $ \mathcal A_i$, then sage $A_i$ indeed guesses hat color correctly.

Up to a cyclic permutation, there are 14 cases of mutual position of the shifted traps
$\mathcal B-3b_{14}$ and $\mathcal C-3c_{14}$. Without loss of generality we may assume that
trap $\mathcal C=\{0,4,8\}$ is not shifted, and trap
$\mathcal B$ has one of the positions $\{3i, 3i+1, 3i+2\}$, $i=0$, 1, 2, \dots,
13. In each of these cases, define sets $\mathcal A_2$, $\mathcal A_3$, $\mathcal
A_{14}$ as specified in the table.

\begin{center}
\footnotesize
\begin{tabular}{cccccc}
$\mathcal C$ & $\mathcal B$ & $\mathcal A_2$ & $\mathcal A_3$ & $\mathcal A_{14}$& Superpositions\\ \hline
0, 4, 8        & 0, 1, 2    & $[5, 25]$ & $[26,39]$ & 40, 41, 3 & 0, 8\\
0, 4, 8        & 3, 4, 5    & $[7, 27]$ & $[28,41]$ & 1, 2, 6   & 4, 8\\
0, 4, 8        & 6, 7, 8    & $[11, 31]$ & $[32,3]$ & 5, 9, 10  & 0, 8\\
0, 4, 8        & 9, 10, 11  & $[15, 35]$ & $[36,7]$ & 12, 13, 14& 0, 4\\
0, 4, 8        & 12, 13, 14 & $[15, 35]$ & $[36,7]$ & 9, 10, 11 & 0, 4\\
0, 4, 8        & 15, 16, 17 & $[21, 41]$ & $[1,14]$ & 18, 19, 20& 4, 8\\
0, 4, 8        & 18, 19, 20 & $[21, 41]$ & $[1,14]$ & 15, 16, 17& 4, 8\\
0, 4, 8        & 21, 22, 23 & $[25, 3] $ & $[5,18]$ & 19, 20, 24& 0, 8\\
0, 4, 8        & 24, 25, 26 & $[29, 7] $ & $[9,22]$ & 23, 27, 28& 0, 4\\
0, 4, 8        & 27, 28, 29 & $[5, 25] $ & $[32,3]$ & 26, 30, 31& 0, 8\\
0, 4, 8        & 30, 31, 32 & $[9, 29] $ & $[33,4]$ & 5, 6, 7   & 0, 4\\
0, 4, 8        & 33, 34, 35 & $[12,32] $ & $[36,7]$ & 9, 10, 11 & 0, 4\\
0, 4, 8        & 36, 37, 38 & $[15,35] $ & $[1,14]$ & 39, 40, 41& 4, 8\\
0, 4, 8        & 39, 40, 41 & $[15,35] $ & $[1,14]$ & 36, 37, 38& 4, 8\\
\end{tabular}
\end{center}

\end{proof}

\subsection{A planar graph with hat guessing number at least
  14}\label{subsec:planar14}

\begin{theorem}\label{thm:planar6}
  There exists a planar graph $G''$ with hat guessing number at least 14.
\end{theorem}

The proof immediately follows from lemma~\ref{thm:wheel} applied to game
$\HG_1=\langle {K_5^-, [2, 3, 14, 14, 14]} \rangle$ and  game $\HG=$``Trefoil'' on an outerplanar graph (fig.~\ref{fig:planar6}).

\begin{proof}
Let game $\HG=$``Trefoil'' (see~fig.~\ref{fig:planar6}), this is a
  game on outerplanar graph $G$ with 13 vertices, denote them by $A_1$,
  $A_2$, \dots, $A_{13}$. Consider 13 copies $\HG_i$, $1\leq i\leq 13$,
  of the game $\langle {K_5^-, [2, 3, 14, 14, 14]} \rangle$. In each of these games,
  denote by $O$ the vertex of hatness 2 and denote by $A_i$ the vertex of hatness 3. It is
  easy to see that graphs $G_i$ are planar. Applying the cone theorem to these
  graphs (see~fig.~\ref{fig:planar2-16}), we obtain a game $\langle {G',h'}
  \rangle$ with one vertex $O$ of hatness 2 and several vertices of hatnesses 14,
  18, or 24. Finally, multiply four copies of game $\langle {G',h'} \rangle$ with respect to the vertex~$O$.

  The crafted game $\HG''=\langle {G'',h''} \rangle$ is played on planar graph $G''$,
  where $G''=G'\gplus{O} G'\gplus{O} G'\gplus{O} G'$ and $\hgn(G'')\geq \min h''=14$.

\end{proof}

\begin{figure}[h]
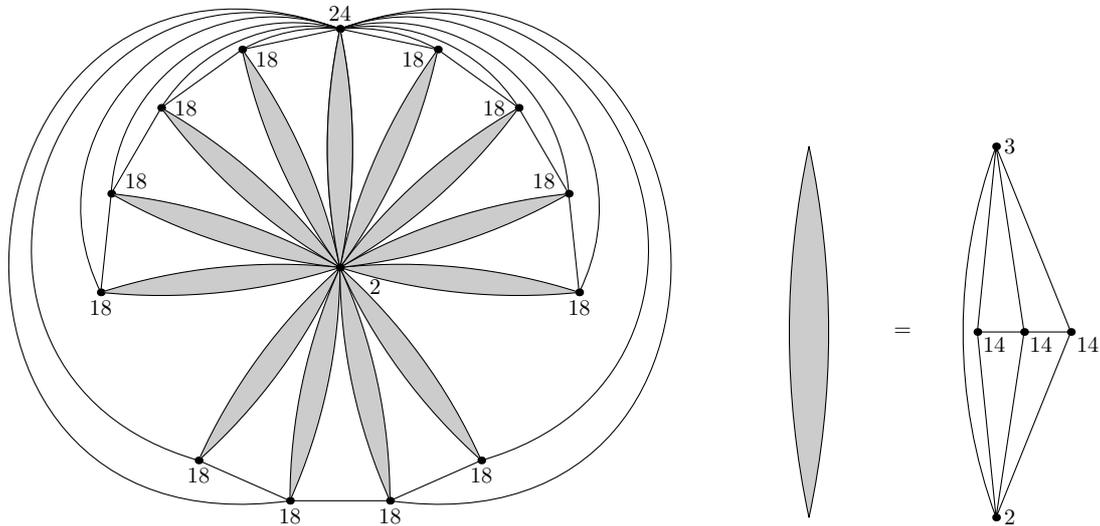
%
  \setlength{\unitlength}{20bp}\footnotesize
  \begin{center}
    \epsfig{scale=1,file=pic/planar1.mps}
    \qquad\qquad
    \epsfig{scale=1,file=pic/planar2.mps}
  \end{center}
  \caption{The game $\langle {G',h'} \rangle$. Each petal stands for a copy of the graph $K_5^{-}$.
  Gluing four copies of this game together, we obtain a planar graph with  $\hgn(G'\gplus{2} G'\gplus{2} G'\gplus{2} G')\geq 14$.}
  \label{fig:planar2-16}
\end{figure}

\section{Conclusion}

The version of the \hats game with non-constant hatness function and the theory of
constructors have proved to be a fruitful approach to study of the classical hat guessing game. They
provide natural formulations, structurally visible examples, visual and at the same
time meaningful proofs. But the computational complexity of the game prevents from
putting forward hasty conjectures and  effectively protects the game from a complete analysis.

\printbibliography

\end{document}